\documentclass[12pt]{article}

\usepackage{amsmath,amsfonts}
\usepackage[authoryear]{natbib}
\RequirePackage[colorlinks,citecolor=blue,urlcolor=blue]{hyperref}
\RequirePackage{hypernat}

\usepackage{amsthm}
\usepackage{fullpage}
\usepackage{times}
\usepackage{graphicx}
\usepackage{color}
\usepackage{multirow}
\usepackage{rotating}
\usepackage{latexsym}

\newcommand{\captionfonts}{\normalsize}

\makeatletter  
\long\def\@makecaption#1#2{%
  \vskip\abovecaptionskip
  \sbox\@tempboxa{{\captionfonts #1: #2}}%
  \ifdim \wd\@tempboxa >\hsize
    {\captionfonts #1: #2\par}
  \else
    \hbox to\hsize{\hfil\box\@tempboxa\hfil}%
  \fi
  \vskip\belowcaptionskip}
\makeatother   

\newcommand{\B}{\mathcal{B}}
\newcommand{\R}{\mathbb{R}}

\newcommand{\E}{\mathbb{E}}
\newcommand{\D}{\mathcal{D}}
\newcommand{\diam}{\mathrm{diam}}
\newcommand{\var}{\mathrm{Var}}

\newcommand{\W}{\mathcal{W}}
\newcommand{\WK}{\mathcal{W}^\infty_{\alpha,K_0}}
\newcommand{\F}{\mathcal{F}}
\newcommand{\bs}{\boldsymbol}
\newcommand{\ess}{\mathrm{ess}}

\newtheorem{theorem}{Theorem}[section]
\newtheorem{definition}{Definition}[section]
\newtheorem{assumption}{Assumption}
\newtheorem{lemma}{Lemma}[section]

\newtheorem{corollary}{Corollary}[section]
\newtheorem{proposition}{Proposition}[section]
\newtheorem{remark}{Remark}

\title{Convergence Rates for Mixture-of-Experts}
\author{Eduardo F. Mendes \and Wenxin Jiang\\{ Department of Statistics, Northwestern University}}

\begin{document}
\maketitle
\begin{abstract}
In mixtures-of-experts (ME) model, where a number of submodels (experts) are combined, there have been two longstanding problems: (i) how many experts should be chosen, given the size of the training data? (ii) given the total number of parameters, is it better to use a few very complex experts, or is it better to combine many  simple experts? In this paper, we try to provide some insights to these problems through a theoretic study on a ME structure where $m$  experts  are mixed, with each expert being related to a polynomial regression model of order $k$.  We study the  convergence rate of the maximum likelihood estimator (MLE), in terms of how fast the Kullback-Leibler divergence of the estimated density converges to the true density, when the sample size $n$ increases.  The convergence rate  is found to be dependent on both $m$ and $k$, and certain choices of $m$ and $k$ are found to produce optimal convergence rates.  Therefore, these results shed light on the two aforementioned important problems: on how to choose $m$, and on how $m$ and $k$ should be compromised, for achieving good convergence rates.
\end{abstract}



\section{Introduction}

Mixture-of-experts models (ME) \citep{jacobsetal1991} and  hierarchical mixture-of-experts models (HME) \citep{jordanjacobs1994} are powerful tools for estimating the density of a random variable $Y$ conditional on a known set of covariates $X$. The idea is to ``divide-and-conquer''. We first divide the covariate space into \textit{subspaces}, then approximate each \textit{subspace} by an adequate model and, finally, weigh by the probability that $X$ falls in each \textit{subspace}. Additionally, it can be seen as a generalization of the classical mixture-of-models, whose weights are constant across the covariate space. Mixture-of-experts have been widely used on a variety of fields including image recognition and classification, medicine, audio classification and finance. Such flexibility have also inspired a series of distinct models including \citet{woodetal2002}, \citet{carvalho2005a}, \citet{gewekekeene2007}, \citet{woodetal2008}, \citet{villanietal2009}, \citet{younghunter2010} and \citet{woodetal2011}, among many others. 

We consider a framework similar to \citet{jiangtanner1999a} among others. Assume each expert is in a one-parameter exponential family with mean $\varphi(h_k)$, where $h_k$ is a $k^{th}$-degree polynomial on the conditioning variables $X$ (hence a linear function of the parameters) and $\varphi(\cdot)$ is the inverse link function. In other words, each expert is a Generalized Linear Model on an one-dimensional exponential family (GLM1). We allow the target density to be in the same family of distributions, but with conditional mean $\varphi(h)$ with $h\in \WK$, a Sobolev class with $\alpha$ derivatives. Some examples of target densities include the Poisson, binomial, Bernoulli and exponential distributions with unknown mean. Normal, gamma and beta distributions also fall in this class if the dispersion parameter is known. 

One might be reluctant to use (H)ME models with polynomial experts since it leads to more and more complex models as the degree $k$ of the polynomials increases. The discussion whether is better to mixture many simple models or fewer more complex models is not new in the literature of mixture-of-experts. Earlier in the literature, \citet{jacobsetal1991} and \citet{pengetal1996} proposed mixtures of many simple models; more recently, \citet{woodetal2002} and \citet{villanietal2009} considered using only a few complex models. \citet{celeuxetal2000} and \citet{geweke2007} advocate for mixing fewer complex models, claiming that mixture models can be very difficult to estimate and interpret. We justify the use of such models through the approximation and estimation errors. We illustrate that might be a gain in a small increase of $k$ compared to the linear model $k=1$ but the number of parameters increases exponentially as $k$ increases. Therefore, a balance between the complexity of the model and the number of experts is required for achieving better error bounds.

This work extends \citet{jiangtanner1999a} in few directions. We show that, by including polynomial terms, one is able to improve the approximation rate on sufficiently smooth classes. This rate is sharp for the piecewise polynomial approximation as shown in \citet{windlund1977}. Moreover, we contribute to the literature by providing rates of convergence of the maximum likelihood estimator to the true density. We emphasize that such rates have never been developed for this class of models and the method used can be straightforwardly generalized to more general classes of mixture of experts. Convergence of the estimated density function to the true density and parametric convergence of the quasi-maximum likelihood estimator to the pseudo-true parameter vector are also obtained.

We found that, under slightly weaker conditions than \citet{jiangtanner1999a}, the approximation rate in Kullback-Leibler divergence is uniformly bounded by $c\times m^{-2[\alpha\wedge (k+1)]/s}$, where $c$ is some constant not depending on $k$ or $m$, and $s$ the number of independent variables. This is a generalization of the rate found in \citet{jiangtanner1999a} who assume $\alpha=2$ and $k=1$. The convergence rate of the maximum likelihood estimator to the true density is $O_p\left(m^{-2[\alpha\wedge (k+1)]/s}+(mJ_k+v_m)n^{-1}\log n\right)$, where $J_k$ is the total number of parameters in each polynomial (typically ${k+s}\mbox{ choose }{k}$), and $v_m$ is the number of parameters in the weight functions. To show the previous results we do not assume identifiability of the model as it is natural for mixture-of-experts to be unidentifiable under permutation of the experts. If we further assume identifiability \citep{jiangtanner1999a,mendesetal2006}, and that the likelihood function has a unique maximizer, we are able to remove the ``$\log n$'' term in the convergence rate. Optimal nonparametric rates of convergence can be attained if $k = \alpha-1$ and $m=O\left(n^{s/(2\alpha+s)}\right)$ \citep{stone1980, stone1985, chen2006}.

\citet{zeevietal1998} show approximation in the $L^p$ norm and estimation error for the conditional expectation of the ME with generalized linear experts.  \citet{jiangtanner1999a} show consistency and approximation rates for the HME with generalized linear model as experts and a general specification for the gating functions. They consider the target density to belong to the exponential family with one parameter. Their approximation rate of the Kullback-Leibler divergence between the target density and the model is $O(1/m^{4/s})$, where $m$ the number of experts and $s$ the number of covariates. \citet{norets2009} show the approximation rate for the mixture of Gaussian experts where both the variance and the mean can be nonlinear and the weights are given by multinomial logistic functions. He considers the target density to be a smooth continuous function and the dependent variable $Y$ to be continuous and satisfy some moment conditions. His approximation rate is $O(1/m^{s+2+1/(q-2)+\varepsilon})$, where $Y$ is assumed to have at least $q$ moments and $\varepsilon$ is a small number. Despite these findings, there are no convergence rates yet for the maximum likelihood estimator of mixture-of-experts type of models in the literature.

By studying the convergence rates in this paper, we will be able to shed light on two long-standing problems in ME: \textit{(i) How to choose the number of experts $m$ for a given sample size $n$? (ii) Is it better to mix many simple experts or to mix a few complex experts?} None of the works discussed above directly address these questions. Our study of a ME structure mixing $m$ of the $k$th order polynomial submodels is particularly useful in studying problem (i), which cannot be studied in the framework of \citep{jiangtanner1999a}, for example, who have restricted to the special case $k=1$.

Throughout the paper we use the following notation. Let $x=(x_1,\dots,x_s)\in S$ and $h(x):S\rightarrow \R$ an $\lambda$ denote some measure. For any finite vector $x$ we use $|x| = \sum_{j=1}^s |x_j|$ and $|x|_p = \left( \sum_{j=1}^p|x_j|^p \right)^{1/p}$, for $p\in[1,\infty)$, if $p=\infty$ we take $|x|_\infty = \sup_{j=1,\dots,s}|x_j|$. For some function $h(x)$ and measure $\lambda$ we denote $\|h\|_{p,S} = \left( \int_S|h|d\lambda \right)^{1/p}$, for $p\in[1,\infty)$, and for $p=\infty$ we have $\|h\|_{\infty,S}=\ess\sup_{x\in S}|h(x)|$. 

The remainder of the paper is organized as follows. In the next section we introduce the target density and mixture of experts models. We also demonstrate that the quasi-maximum likelihood estimator converges to the pseudo-true parameter vector. Section \ref{s:main} establishes the main results of the paper: approximation rate, convergence rate and non-parametric consistency. Section \ref{s:discussion} discusses model specification and the tradeoff that we unveil between the number of experts and the degree of the polynomials. In the concluding remarks we compare our results with \citet{jiangtanner1999a} and provide direction for future research. The appendix collects technical details of the paper and a deeper treatment on how to bound the estimation error.

\section{Preliminaries}

In this section we introduce the target class of density, mixture-of-experts model with GLM1 experts and the estimation algorithm.

\subsection{Target density}\label{s:target}
Consider a sequence of random vectors $\{(X_i',Y_i)'\}_{i=1}^n$ defined on $( (\Omega \times A)^n,\B_{ (\Omega \times A)^n},P^n_{xy})$ where $X\in\Omega\subset\R^s$, $Y\in A\subseteq \R$ and $\B_{S}$ is the Borel $\sigma$-algebra generated by the set $S$. We assume that $P_{xy}$ has a density $p_{xy} = p_{y|x}p_x$ with respect to some measure $\lambda$. More precisely, we assume that  $p_x$ is known and $p_{y|x}$ is member of an one-dimensional exponential family, i.e.
\begin{equation}
	p_{y|x} = \exp\left\{ ya(h(x)) + b(h(x)) + c(y) \right\},
	\label{eq:targetdensity}
\end{equation}
where $a(\cdot)$ and $b(\cdot)$ are known three times continuously differentiable functions, with first derivative bounded away from zero and $a(\cdot)$ has a non-negative second derivative; $c(\cdot)$ is a known measurable function of $Y$. The function $h(\cdot)$ is a member of $\WK(\Omega)$, a Sobolev class of order $\alpha$\footnote{Suppose $1\le p \le \infty$ and $\alpha>0$ is an integer. We define $\W^\alpha_{p,K_0}(\Omega)$ as the collection of measurable functions $h$ with all distributional derivatives $D^r f$, $|r|\le \alpha$, on $L^p(\Omega)$, i.e. $\|D^rh\|_{p,\Omega}\le K_0$. Here $D^r = \partial^{|r|}/(\partial^{r_1}x_1\dots\partial^{r_s}x_s)$ and $|r| = r_1 + \dots + r_s$ for $r=(r_1, \ldots,r_s)$.}. Throughout the paper denote by $\Pi(\WK)$ the class of density functions $p_{xy}=p_{y|x}p_x$.
	
The one-parameter exponential family of distributions includes the Bernoulli, exponential, Poisson and binomial distributions, it also includes the Gaussian, gamma and Weibull distributions if the dispersion parameter is known. It is possible to extend the results to the case where the dispersion parameter is unknown, but defined in a compact subset bounded away from zero. In this work we focus only in the one-parameter case.

Some properties of the one parameter exponential family are : (i) conditional on $X=x$, the moment generating function of $y$ exists in a neighborhood of the origin implying that moments of all orders exist; (ii) for each positive integer $j$, $\mu_{(j)}(h)=\int_Ay^j\exp[a(h)y+b(h)+c(y)]d\lambda$ is a differentiable function of $h$; and (iii) the first conditional moment $\mu_{(1)}(h) = -\dot{b}(h)/\dot{a}(h) = \varphi(h)$, where $\dot{a}(h)$ and $\dot{b}(h)$ are the first derivatives of $a(h)$ and $b(h)$ respectively, and $\varphi(\cdot)$ is called the inverse link function. See \citet{lehmann1991} and \citet{glm} for more results about the exponential family of distributions.

\subsection{Mixture-of-experts model} \label{s:model}
The mixture-of-experts model with GLM1 experts is defined as:
\begin{align}
	f_{m,k}(x,y;\zeta) &= \sum_{j=1}^m g_j(x;\nu)\,\pi(h_k(x,\theta_j),y)\cdot p_x\nonumber\\
	\quad &= \sum_{j=1}^m g_j(x;\nu)\,\exp\{ya(h_k(x;\theta_j)) + b(h_k(x;\theta_j)) + c(y)\}\cdot p_x,
	\label{eq:approxclass}
\end{align}
where the functions $g_j \ge 0$ and $\sum_{j=1}^mg_j = 1$ with parameters $\nu\in V_m\subset\ell^2(\R^{v_m})$\footnote{We denote $\ell^2(\R^k) \equiv \{x\in\R^k:\,\sum_{j=1}^kx_j^2 <\infty\}$. }, $v_m$ denoting the dimension of $\nu$. The functions $h_k(x;\theta_j)$ are $k^{th}$-degree polynomials on $\Omega$ with parameter vector $\theta_j \in\Theta_k \subset\ell^2(\R^{J_k})$, $J_k$ denoting the dimension of $\theta_j$; write the vector of parameters of all experts as $\theta = (\theta_1',\dots,\theta_m)'$ defined on $\Theta_{mk} \equiv \Theta_k^m$. The parameter vector of the model is $\zeta = (\nu',\theta')'$ and is defined on $V_m\times\Theta_{mk}$, a subset of $\R^{v_m\times mJ_k}$. Throughout the paper we denote by $\F_{m,k}$ the class of (approximant) densities $f_{m,k}$.

To derive consistency and convergence rates, one need to impose some restrictions on the functions $\pi$ and $g_j$ to avoid abnormal cases. This condition is not restrictive and is satisfied by the multinomial logistic weight functions ($g's$) and the Bernoulli, binomial, Poisson and exponential experts, among many other classes of distributions and weight functions.
\begin{assumption}
	There exist functions $c_g(x) = (c_g^{(1)}(x),\dots,c_g^{(v_m)}(x))'$ and $F(x,y) = (F^{(1)}(x,y),$ $\dots,F^{(J_k)}(x,y))'$ with $\E[c_g(X)'c_g(X)] < \infty$ and $\E[F(X,Y)'F(X,Y)]<\infty$, such that	the vector-function $g(x;\nu) = (g_1(x;\nu),\dots,g_m(x;\nu))'$ satisfy
	\[
	\sup_{\nu\in V_m}\frac{\partial \log {g(x;\nu)}}{\partial\nu_i}\leq c^{(i)}_g(x);
	\]
	and each expert $\pi(h_k(x;\theta_j),y)$ satisfy
	\[
	\sup_{\theta\in\Theta_{k}}\frac{\partial\pi(h_k(x,\theta_{j}),y)}{\partial\theta_{ji}}\leq F^{(i)}(x,y), \quad\mbox{ for each }1\leq j\leq m.
	\] 
	
	\label{a:bound_df}
\end{assumption}

\subsection{Maximum likelihood estimation and the EM algorithm}\label{mle}

\subsubsection{Maximum likelihood estimation}
We consider the maximum likelihood method of estimation. We want to find the parameter vector $\hat{\zeta}_n = (\hat{\nu}_n',\hat{\theta}_n')'$ that maximizes
\begin{equation}
	L_n(\zeta) = n^{-1}\sum_{i=1}^n \log\left\{ f_{m,k}(X_i,Y_i;\zeta)/\varphi_0(X_i,Y_i) \right\},
	\label{eq:lkl}
\end{equation}
where $\varphi_0(X,Y) = \exp(c(Y))p_x(X)$. That is, 
\begin{equation}
	\hat{\zeta}_n = \arg\max_{\zeta\in V_m\times\Theta_{mk}}L_n(\zeta).
	\label{eq:mle}
\end{equation}
The maximum likelihood estimator is not necessarily unique. In general, mixture-of-experts models are not identifiable under permutation of the experts. To circumvent this issue one must impose restrictions on the experts and the weighting (or the parameter vector of the model), as shown in \citet{jiangtanner1999c}.

Define the Kullback-Leibler (KL) divergence between $p_{xy}$ and $f_{m,k}$ as
\begin{equation}
	KL(p_{xy},f_{m,k}) = \int_\Omega\int_A \log\frac{p_{xy}}{f_{m,k}}dP_{y|x}dP_x.
   \label{eq:kl}
\end{equation} 
The log-likelihood function in (\ref{eq:lkl}) converges to its expectation with probability one as the number of observations increases. Therefore, in the limit, the minimizer $\hat{f}_{m,k}$ of \eqref{eq:lkl} (indexed by $\hat{\zeta}_n$) also minimizes the Kullback-Leibler divergence between the true density and the estimated density.

In this work only consider i.i.d. observations but is straightforward to extend the results to more general data generating processes. Next assumption formalizes it.

\begin{assumption}[Data Generating Process]
	The sequence $(X_i,Y_i)_{i=1}^n$, $n=1, 2, \dots$ is an independent and identically distributed sequence of random vectors with common distribution $P_{xy}$.
	\label{a:dgp}
\end{assumption}

Next results ensures the existence of such estimator. 
\begin{theorem}[Existence]
	For a sequence $\{(V_m\times\Theta_{mk})_n\}$ of compact subsets of $V_m\times\Theta_{mk}$, $n=1,2,\cdots$, there exists a $\B(\Omega\times A)-$measurable function $\hat\zeta_n:\Omega\times A \rightarrow (V_m\times\Theta_{mk})_n$, satisfying equation (\ref{eq:mle}) $P_{xy}$-almost surely.
	\label{thm:existence}
\end{theorem}

We demonstrate that under the classical assumptions, such as identifiability and unique maximizer, the maximum likelihood estimator $\hat{\zeta}$ consistently estimate the best model in the class $\mathcal{F}_{mk}$ indexed by $\zeta^*$, i.e. the maximum likelihood estimator $\hat\zeta$ converges almost surely to $\zeta^*$.  It can be shown that the convergence results also hold for the ergodic case if we assume that $(\log f_{m,k}(X_i,Y_i;\zeta))_{i=1}^n$ is ergodic. However, simpler conditions to ensure ergodicity of the likelihood function are not trivial and hence out of the scope of this paper. 

\begin{assumption}[Identifiability]
	For any distinct $\zeta_1$ and $\zeta_2$ in $V_m\times\Theta_{mk}$, for almost every $(x,y)\in \Omega\times A$,  
	\[
	f_{m,k}(x,y;\zeta_1) \neq f_{m,k}(x,y;\zeta_2)
	\]
	\label{a:ident}
\end{assumption}

 \citet{jiangtanner1999c} find sufficient conditions for identifiability of the parameter vector for the HME with one layer, while \citet{mendesetal2006} for a binary tree structure. Both cases can be adapted to more general specifications. Although one can show consistency to a set, we adopt a more traditional approach requiring identifiability of the parameter vector. 

\begin{assumption}[Unique Maximizer]
	Let $\zeta=(\nu',\theta')'$ and $\zeta^*$ the argument that maximizes $\E\log f_{m,k}$ over $V_m\times\Theta_{mk}$. Then
	\begin{equation}
		\det\left( \E\frac{\partial^2}{\partial \zeta\partial\zeta'}\log f_{m,k}\vert_{\zeta=\zeta^*} \right) \neq 0
		\label{eq:uniquemax}
	\end{equation}
	\label{a:uniquemax}
\end{assumption}

This assumption follows from a second order Taylor expansion of the expected likelihood around the parameter vector that maximizes \eqref{eq:kl}, denoted $\zeta^*$. We require the Hessian to be invertible at $\zeta^*$. The requirement for an identifiable unique maximizer is only technical in a sense that the objective function is not allowed to become too flat around the maximum (For more discussion on this topic see \citet{bateswhite1985}, pg 156, and \citet{qmle} chapter 3). A similar assumption was made in the series of papers from \citet{carvalho2005a, carvalho2005b, carvalho2006, carvalho2007} and \citet{zeevietal1998} and is an usual assumption in the estimation of misspecified models. 

\begin{theorem}[Parametric consistency of misspecified models]
	Under Assumptions \ref{a:bound_df}, \ref{a:dgp}, \ref{a:ident}, and \ref{a:uniquemax}, the maximum likelihood estimate $\hat\zeta \rightarrow \zeta^*$ as $n\rightarrow\infty$ $P_{xy}$-a.s.
	\label{thm:param_consistency}
\end{theorem}

\citet{huerta2003} and the series of papers by \citet{carvalho2005a, carvalho2005b, carvalho2006, carvalho2007} derive similar results for time series processes. 

\subsubsection{The EM algorithm}
It is often easier to maximize the \textit{complete likelihood function} of a (H)ME instead of (\ref{eq:lkl}) (see \citet{jordanjacobs1994}, \citet{xujordan1996} and \citet{yangma2011}). Let $z_i' = (z_{i1},\cdots,z_{im})$ denote a binary vector with $z_{ij} = 1$ if the observation $(x_i,y_i)$ is generated by the expert $j$ (i.e. $\pi(h_k(\cdot,\theta_j),\cdot)$). We assume $z_i$ has a multinomial distribution with parameters $\tau_i' = (\tau_{i1},\cdots,\tau_{im})$. The complete log-likelihood function is given by
\begin{equation}
   l_n^c(\kappa) = \sum_{i=1}^n \sum_{j=1}^mz_{ij}\left( \log g_i(x_i,\nu)+\log\pi(h_k(x_i;\theta_j),y_i) - \log\varphi_0(x_i,y_i) \right),
   \label{eq:complkl}
\end{equation}
where $\kappa = (\nu',\theta',\tau')'$. 

We can estimate this model using the expectation-maximization (EM) algorithm put forward by \citet{dempsteretal1977}. Let $\kappa^{(l)}=(\nu^{(l)},\theta^{(l)},\tau^{(l)})$ denote the parameter estimates at the $l$th iteration and define $q(\kappa;\kappa^{(l)}) = \E(l_n^c|x,y;\kappa^{(l)})$. In the E-step, we obtain $q(\kappa,\kappa^{(l)})$ by replacing $z_{ij}$ with its expectation 
\begin{equation}
   \tau_{ij}^{(l)} = \frac{g_j(x_i,\nu^{(l)})\pi(h_k(x_i;\theta_j^{(l)}),y_i)}{\sum_{j=1}^{m}g_j(x_i;\nu^{(l)})\pi(h_k(x_i;\theta_j^{(l)}),y_i)}.
   \nonumber
\end{equation}

In the M-step we maximize $q(\kappa;\kappa^{(l)})$ with respect to $\nu$ and $\theta$. The problem simplifies to find the parameters $\nu^{(l+1)}$ that maximize
\begin{equation}
   q(\nu;\kappa^{(l)}) = \sum_{i=1}^n\tau_{ij}^{(l)}\log g_j(x_i;\nu),
   \label{eq:emg}
\end{equation}
and to find the parameters $\theta^{(l+1)}$ we have to maximize
\begin{equation}
   q(\theta;\kappa^{(l)}) = \sum_{i=1}^{n}\sum_{j=1}^{k}\tau_{ij}^{(l)}[y_ia(h_k(x_i;\theta_j))+b(h_k(x_i;\theta_j))].
   \label{eq:empi}
\end{equation}

\section{Main results} \label{s:main}
In this section we present the main results of the paper. Write the KL-divergence as follows:
\begin{equation}
	KL(p_{xy},f_{m,k}) = KL(p_{xy},f_{m,k}^*) + \E\left[ \log\frac{f_{m,k}^*}{f_{m,k}} \right],
	\label{eq:kl-split}
\end{equation}
where $f_{m,k}^*$ is the minimizer of the minimizer of $KL(p_{xy},f_{m,k})$ on $\mathcal{F}_{m,k}$. The first term in the right-hand side is the \textbf{approximation error} and the second term is the \textbf{estimation error}. The approximation error measures ``how well'' an element of $\mathcal{F}_{m,k}$ approximates $p_{xy}$, and approaches zero as $m$ increases. The estimation error measures ``how far'' is the estimated model from the best approximant in the class. Our goal is to find bounds for both approximation and estimation errors and combine these results to find the convergence rate of the maximum likelihood estimator.

\subsection{Approximation rate}\label{approxrate} 
We follow \cite{jiangtanner1999a} to bound the approximation error. Define the \textit{upper divergence} between $p\in\Pi(\WK)$ and $f_{m,k}\in\mathcal{F}_{m,k}$ as
\begin{equation}
	\D(p,f_{m,k}) = \int_\Omega \sum_{j=1}^mg_j(x,\nu)(h_k(x;\theta_j)-h(x))^2dP_x.
	\label{eq:updiv}
\end{equation}

We can use the upper divergence to bound the KL-divergence.
\begin{lemma}
	Let $p\in\Pi(\WK)$ and $f_{m,k}\in\mathcal{F}_{m,k}$. If $\ess\sup_{x\in\Omega}|h(x)| < \infty$,
	\begin{equation}
		KL(p,f_{m,k})\leq M_{\infty}\mathcal{D}(p,f_{m,k})
		\nonumber
	\end{equation}
	where $M_\infty \geq (1/2)\ess\sup_{x\in\Omega}[|\varphi(h(x))|\cdot |\ddot{a}(h(x))| + |\ddot{b}(h(x))|]$.
	\label{l:updiv}
\end{lemma}
This lemma will be used to bound uniformly the approximation rate of the family of functions $\mathcal{F}_{m,k}$.

Before presenting the main conditions, we shall introduce some key concepts.
\begin{definition}[Fine partition]
	For $m=1,2,\dots$, let $Q^m = \{Q_j^m\}_{j=1}^{r_m}$ be a partition of $\Omega$. If $m\rightarrow\infty$ and if for all $x_1,x_2 \in Q_j^m$, $\max_{1\leq i\leq s}|(x_1-x_2)_i|\leq c_0/r_m^{1/s}$, for some constant $c_0$ independent of $x_1$, $x_2$, $m$ or $j$. Then $\{Q^m,\,m=1,2,\dots\}$ is called a sequence of fine partitions with cardinality $r_m$ and bounding constant $c_0$.
\end{definition}

Here we use some abuse of notation by using $m$ as an index of the collection of partitions of $\Omega$. However, this abuse of notation is justified because $m$ is an increasing sequence and the collection of partitions depends on an increasing function of $m$. The next definition will be useful later to bound the ``growth rate'' of the model and is useful to deal with hierarchical mixture of experts (see \citet{jiangtanner1999a}).

\begin{definition}[Subgeometric]
   A sequence of numbers $a_j$ is called sub-geometric with rate bounded by $M_1$ if $a_j\in\mathbb{N}$, $a_j\rightarrow\infty$ as $j\rightarrow\infty$, and $1<|a_{j+1}/a_j|<M_1$ for all $j=1,2,\dots$ and for some finite constant $M_1$.
\end{definition}

The key idea behind find the approximation rates, is to control the approximation rate inside each fine partition of the space. More precisely, bound the approximation inside the ``worst'' (more difficult to approximate) partition. We need the following assumption.

\begin{assumption}
	There exists a fine partition $Q^m$ of $\Omega$, with bounding constant $c_0$ and cardinality sequence $r_m$, $m=1,2,\cdots$, such that $\{r_m\}$ is sub-geometric with rate bounded by some constant $M_1$, and there exists a constant $c_1>0$, and a parameter vector $\nu_{c_1} \in V_m$ such that
   \begin{equation}
		\max_{1\leq j \leq r_m} \|g_j(\cdot;\nu_{c_1})-I_{Q^m_j}(\cdot)\|_{1,\lambda}\leq\frac{c_1}{r_m}.
      \label{eq:a1}
   \end{equation}
	\label{a:finepart}
\end{assumption}

This assumption is similar, but weaker than, the one employed in \citet{jiangtanner1999a} and requires that the vector $g =  (g_1,\cdots,g_{r_m})$ approximates the vector of characteristic functions $(I_{Q^m_1},\cdots,I_{Q^m_{r_m}})$ at a rate not slower then $O(r_m)$.

The notation $r_m$ is introduced to deal with the hierarchical mixture of experts structure. To allow more flexibility define $r_m$ as the maximum number of experts the structure can hold, e.g. a binary tree with $l=1,2,\dots$ layers has at most $2^l$ experts, and if we increase the number of layers by one, the actual number of experts is somewhere between $2^l$ and $2^{l+1}-1$ (here we are assuming the tree is balanced without loss of generality). If we denote this class of models by $\mathcal{F}_{r_m,k}^*$, then $\mathcal{F}_{r_m,k}^*\subseteq \mathcal{F}_{m,k}^*\subset\mathcal{F}_{r_m+1,k}^*$. The sub-geometric assumption ensures that $r_m\leq m < r_{m+1}$, where $m$ is the actual number of experts in the model.

\begin{theorem}[Approximation rate]
	Let $p\in\Pi(\WK)$ and $f_{m,k}\in\mathcal{F}^*_{m,k}$. If assumption \ref{a:finepart} holds, then
	\begin{equation}
		\sup_p \inf_{f_{m,k}}KL(p,f_{m,k}) \leq \frac{c}{m^{2(\alpha\wedge (k+1))/s}}
		\label{eq:approx_kl}
	\end{equation}
	for some constant $c$ not depending on $m$ or $k$.
	\label{thm:approxrate}
\end{theorem}

This result is a generalization of \citet{jiangtanner1999a} in two directions. First we allow the target function to be in a Sobolev class with $\alpha$ derivatives; second, we consider a polynomial approximation to the target function in each experts (in fact, their result is a special case when $\alpha=2$ and $k=1$). This generalization enables us to address the important problem: whether it is better to mix many simple experts or to mix a few complex experts.  The result also holds under more general specifications of densities/experts. In the case we also have a dispersion parameter to estimate, we just have to modify the lemma \ref{l:updiv} accordingly and the same result holds. 

This rate also agrees with the optimal approximation rate of functions on $\WK$ by piecewise polynomials \citep{windlund1977}. One can see that, under assumption \ref{a:finepart}, it is exactly what we are doing. Therefore this approximation rate is sharp.

\subsection{Convergence rate}

In this section we deduce the convergence rate for the mixture-of-experts model. Equation (\ref{eq:kl-split}) gives us an expansion of the KL divergence in terms of the approximation and estimation errors. In the previous section we found a bound for the approximation error, in this section we will find the estimation error and combine with the approximation error to find the rate of convergence.

The estimation error is the ``how far'' is the estimated function from the best approximant in the class. We will demonstrate that the estimation error  in (\ref{eq:kl-split}) is $O_p( (mJ_k+v_m)(\log n/n))$. We also show that by combining this result with the approximation rate it is possible to achieve a convergence rate of $O_p( (\log n/n)^{2\tau/(2\tau+s)})$, with $\tau=\alpha\wedge(k+1)$, which is close to the optimal nonparametric rate if $\tau=\alpha$. Moreover, if there is an unique identifiable maximizer to the likelihood problem (assumptions \ref{a:ident} and \ref{a:uniquemax}), we are able to remove the ``$\log n$'' term and achieve a better convergence rate, possibly optimal if $\tau=\alpha$.

The next theorem summarizes the convergence rate of the maximum likelihood estimator $\hat{f}_{m,k}$ with respect the KL divergence between the true density $p_{xy}$ and the estimated density.
\begin{theorem}[Convergence Rate]
	Let $p_{xy}\in\Pi(\WK)$ and $\hat{f}_{m,k}$ denote its maximum likelihood estimator on $\mathcal{F}_{m,k}$. Let $m$ be allowed to increase such that $m\rightarrow \infty$ and $m(\log n/n)\rightarrow 0$ as $n$ and $m$ increase. Under Assumptions \ref{a:bound_df}, \ref{a:dgp} and \ref{a:finepart},  
	\begin{equation}
		KL(p_{xy},\hat{f}_{m,k}) = O_p\left( \frac{1}{m^{2\tau/s}}+(mJ_k+v_m)\frac{\log n}{n} \right),
		\label{eq:convrate}
	\end{equation}
	where $\tau=\alpha\wedge (k+1)$.	In particular, if we assume $v_m=O(m)$, and let $m$ be proportional to $(n/\log n)^{s/(2\tau+s)}$ then
	\begin{equation}
		KL(p_{xy},\hat{f}_{m,k}) = O_p\left( \left( \frac{\log n}{n} \right)^{\frac{2\tau}{2\tau+s}} \right).
		\label{eq:optimrate}
	\end{equation}
	
	\label{thm:convrate}
\end{theorem}

Although the previous result is derived for the i.i.d. case, the result also holds for more general data generating process. In this result we use (through \citet{m-estimation}), an uniform probability inequality for i.i.d. processes to derive theorem \ref{thm:sieves-convrate}, but the same result can be obtained by using uniform inequalities for more general processes. This convergence rate is close to the optimal rate found in the sieves literature if $\tau=\alpha$, see for instance \citet{stone1980} and \citet{barronsheu1991}.

To derive this rate we do not assume that there is an unique identifiable maximizer $f^*_{m,k}$; in fact, we assume $f^*_{m,k}$ is any of such maximizers. The price to pay for such generality is the inclusion of the ``$\log n$'' term in the convergence rates. If we assume $f^*_{m,k}$ is unique and uniquely identified by a parameter vector $\zeta^*$, we can explore the localization property of theorem \ref{thm:sieves-convrate}. More precisely, we can explore the fact that we are only interested in the behavior of the empirical process around a neighborhood of $f^*_{m,k}$. Under such conditions and assuming $\tau=\alpha$, we are able to achieve the optimal convergence rate in the sieves literature \citep{stone1980,barronsheu1991}.

\begin{theorem}[Optimal Convergence Rate]
	Let $p_{xy}\in\Pi(\WK)$ and $\hat{f}_{m,k}$ denote its maximum likelihood estimator on $\mathcal{F}_{m,k}$. Let $m$ be allowed to increase such that $m\rightarrow \infty$ and $m/n\rightarrow 0$ as $n$ and $m$ increase. Under Assumptions \ref{a:bound_df}--\ref{a:finepart},  
	\begin{equation}
		KL(p_{xy},\hat{f}_{m,k}) = O_p\left( \frac{1}{m^{2\tau/s}}+\frac{(mJ_k+v_m)}{n} \right),
		\label{eq:op_convrate}
	\end{equation}
	where $\tau=\alpha\wedge (k+1)$.	In particular, if we assume $v_m=O(m)$, and let $m$ be proportional to $n^{s/(2\tau+s)}$ then
	\begin{equation}
		KL(p_{xy},\hat{f}_{m,k}) = O_p\left( n^{-\frac{2\tau}{2\tau+s}} \right).
		\label{eq:op_optimrate}
	\end{equation}
	
	\label{thm:op_convrate}
\end{theorem}

The same result follows for more general data generating processes and the same considerations after theorem \ref{thm:convrate} hold. 

By imposing there exist an unique maximum we are able to remove the $\log n$ term and recover the optimal convergence rate for sieves estimates found in the literature.

\subsection{Consistency}
Now we apply the previous results to show the maximum likelihood estimator is consistent, i.e. the KL divergence between the true density and the estimated model approaches zero as the sample size $n$, and the index of the approximation class $m$ goes to infinity. Here we show consistency essentially by using the previous results.
\begin{corollary}[Consistency]\label{cor:consistency}
	Let $p_{xy}\in\Pi(\WK)$ and $\hat{f}_{m,k}$ denote its maximum likelihood estimator on $\mathcal{F}_{m,k}$. Allow $m\rightarrow \infty$ and $m(\log n/n)\rightarrow0$ as $n$ and $m$ increase. Under Assumptions \ref{a:bound_df}, \ref{a:dgp} and \ref{a:finepart} , $KL(p_{xy},\hat{f}_{m,k})\rightarrow 0$ as $n$ and $m$ increase.
\end{corollary}

\section{Effects of $m$ and $k$}\label{s:discussion}
We consider a framework similar to \citet{jiangtanner1999a}, but one is allowed to mix $m$ GLM1 experts whose terms are polynomials on the variables, as opposed to $k=1$. We also assume that the true mean function is $\varphi(h)$ with $h\in\WK$, a Sobolev class with $\alpha$ derivatives, as opposed to $\alpha=2$. 

By deriving a convergence rate such as \eqref{eq:op_convrate} in this framework, we are able to gain insight on the two important problems in the area of ME: (i) What number of experts $m$ should be chosen, given the size $n$ of the training data? (ii) Given the total number of parameters, is it better to use a few very complex experts, or is it better to combine many  simple experts? 

For question (i), the results in Theorem \ref{thm:op_convrate} and Corollary \ref{cor:consistency} suggest that good results can be obtained by choosing the number of experts $m$ to grow as  $n^r$ with some power $r\in(0,1)$, which may depend on the dimension of the input space and the underlying smoothness of the target function.  Smoother target functions and lower dimensions generally encourage us to use less experts. 

Question (ii) requires a more detailed study. The complexity of the experts (submodels) are related to $k$, the order of the polynomials.  We see that increasing  $k$ does improve the approximation rate, however this improvement is bounded by the number of derivatives $\alpha$ of the function $h$. Moreover, this approximation rate is known to be sharp for piecewise polynomials \citep{windlund1977}.  The price to pay for this increase in the approximation rate is a larger number of parameters in the model, i.e. a worse estimation error.  We will provide below a theoretical result on the optimal choice of $k$, as well as some numerical evidence.

First of all, an easier expression of the  upper bound of the KL divergence in \eqref{eq:op_convrate} can be derived as $KL\leq O_p(U)$ where $U\equiv (m^{-2(\xi\wedge \alpha)/s}+(m\xi^s)/n) $, where $\xi=k+1$. [This assumes that $v(m)=O(m)$ and uses the fact that the number of parameters needed in $s$-dimensional polynomials of order $k$ is bounded by $J_k\leq (k+1)^s$.]

We now study the upper bound $U$, fixing the product $m\xi^s=C$, where $C$ may depend on $n$ and is a bound for the rough order of the total number of parameters.

\begin{proposition}\label{pro:choosemandk}
Let $\xi=k+1$ and  $U\equiv (m^{-2(\xi\wedge \alpha)/s}+(m\xi^s)/n)$  (\dag) (which is an upper bound for the $KL$ convergence rate derived in Theorem 3.3). Then the following statements are true:

I. Fixing the product $m\xi^s=C$, $U$ is minimized at $\xi=\alpha\wedge(C^{1/s}/e)\equiv\xi_o$.  The corresponding optimal $m=\max(e^s, C/\alpha^s)$.

II. If $\alpha$ is finite, then $U$ achieves the optimal rate $n^{-2\alpha/(s+2\alpha)}$ under the following choices: $\xi$ is any constant  that is at least $\alpha$ and does not vary with $n$, and $m\approx c_1 n^{s/(s+2\alpha)}$ for any constant $c_1>0$.

III. If $\alpha =\infty$, the following choices will make $U$ to have a near-parametric rate $U=O((\ln n)^s/n)$: $m\geq 2$ and is constant in $n$, $\xi \approx c_2\ln n$ for any constant $c_2\geq s$.

\end{proposition}

\begin{remark}\label{rem:1}
This Proposition suggests that  for achieving   optimal performance, the $\xi$ (or $k$, related to the complexity of the experts) and the   $m$ (the number of experts) should be  compromised.  Fixing an upperbound $C$ of the total number of parameters, the optimal $\xi=\alpha\wedge(C^{1/s}/e)\equiv\xi_o$. The optimal compromise therefore depends both on $\alpha$ (smoothness of the target function) and $s$ (the dimension of the input space).  The formula implies that (a) a smoother  target function (indexed by a larger $\alpha$) will favor more complex submodels (with larger $\xi$ or $k$), (b) for a very smooth target function (with large enough $\alpha$), a higher dimension $s$ of the input space will favor  the use of simpler submodels (with smaller $\xi$ or $k+1$, {\it possibly smaller than $\alpha$}) and the use of more experts (bigger $m$).
\end{remark}
\begin{remark}
Although Result I shows how to construct an $exactly$-optimal compromise between $\xi$ and $m$, Results II and III show that good convergence rates are quite robust against deviations from these optimal solutions.   {\it We note that $near$-optimal convergence rates can always be achieved with $\xi$ not being too large compared to the sample size $n$.} This is summarized in the two situations in Results II and III, where we see that even in the case $\alpha=\infty$, we only need about $\xi\sim \ln n$ for us to achieve a near-parametric convergence rate.
\end{remark}

One drawback of the above theoretic analysis is that it has used a rough upper bound (which has a simple expression) for the total number of parameters associated with $k$th order $s$ dimensional polynomials.  Below we conduct some numeric study, where the exact number of parameters are used.  When considering the choice of $k$, a first impulse is to use polynomials of order $a - 1$, but the number of variables in the model increase exponentially with $k$ if $s > 1$. In fact, in many cases it is preferable to use a smaller $k$ and many experts $m$ if one wishes to control the size of the estimation error. This is consistent with the earlier Remark \ref{rem:1} we made for our theoretic analysis.  

Table \ref{tbl:approx} compares the approximation error using distinct values of $k$ and $m$ holding the estimation error fixed. Assume we have $s=5$ variables and $\alpha=6$, a modeler builds a model with $m=5$ experts and, since it is known that $\alpha=6$, also chooses $k=5$. If we further assume $v_m=m\times s$, the total number of parameters in the model is $mJ_k+v_m = 1285$. We can see the smallest approximation error is achieved at $k=3$ and $m=21$.

\begin{table}
	\caption{This table compares the approximation error of the model holding the estimation error fixed. We assume $\alpha=6$ and $s=5$ and allow for distinct specifications of $m$ and $k$.}	
	\centering
	\begin{tabular}[h]{cccc}
	\hline
	$\bs{k}$&$\bs{m}$&$\bs{m^{-2(k+1)/s}}$&$\bs{mJ_k+v_m}$\\
	\hline
	0       & 214    & 0.1169             & 1,284\\
	1       & 117    & 0.0221             & 1,287\\
	2       & 49     & 0.0094             & 1,274\\
	3       & 21     & 0.0077             & 1,271\\
	4       & 10     & 0.0100             & 1,310\\
	5       & 5      & 0.0210             & 1,285\\
	\hline
\end{tabular}
\label{tbl:approx}
\end{table}

Similarly, fixing the approximation error we see that a balance between $m$ and $k$ is necessary. Fix $\alpha=6$ and $s=5$ and assume one wants a model with approximation error proportional to $0.01$. Table \ref{tbl:estim} shows that the model with smaller estimation error that achieves this approximation error is the one with $k=3$ and $m=18$.
\begin{table}[h]
	\caption{This table compares the number of parameters of the model holding the approximation error fixed. We assume $\alpha=6$, $s=5$  and the estimation error to be proportional to $0.01$. We allow for distinct specifications of $m$ and $k$.}	
	\centering
\begin{tabular}[h]{cccc}
	\hline
	$\bs{k}$&$\bs{m}$&$\bs{m^{-2(k+1)/s}}$&$\bs{mJ_k+v_m}$\\
	\hline
	0       & 100,000& 0.0100             & 600,000\\
	1       & 316    & 0.0100             & 3,476\\
	2       & 46     & 0.0101             & 1,196\\
	3       & 18     & 0.0098             & 1,098\\
	4       & 10     & 0.0100             & 1,310\\
	5       & 7      & 0.0099             & 1,799\\
	\hline
\end{tabular}
\label{tbl:estim}
\end{table}

This quick exercise illustrated one of the main conclusions of this paper: it is not true that one should always use few complex models (small $m$ and large $k$) or always choose for many complex ones (small $k$ and large $m$); a balance between $k$ and $m$ should be used instead. Moreover, a small increase in $k$ comparing to the linear model ($k=1$) can have a good improvement on the approximation and estimation errors.

The results in this paper focus only on target density and mixture-of-experts specified in sections \ref{s:target} and \ref{s:model} respectively. However, similar results can be derived for more complex models and target densities. 

\section{Conclusion}
In this paper we study the mixture-of-experts model with experts in an one-exponential family with mean $\varphi(h_k)$, where $h_k$ is a $k^{th}$ order polynomial and $\varphi(\cdot)$ is the inverse link function. We derive sharp approximation rates with respect to the Kullback-Leibler divergence and convergence rate of the maximum likelihood estimator to densities in an one-parameter exponential family with mean $\varphi(h)$ with $h\in\WK$, a Sobolev class with $\alpha$ derivatives. We found that the convergence rate of the maximum likelihood estimator to the true density is $O_p\left(m^{-2[\alpha\wedge(k+1)]/s}+ (mJ_k+v_m) n^{-1}\log n\right)$, where $n$ is the number of observations, $s$ is the number of covariates, $J_k$ is the number of parameters of the polynomial $h_k$, $m$ the number of experts and $v_m$ is the number of parameters on the weight functions. Further, if the maximum likelihood estimator is uniquely identified we can remove the ``$\log n$'' term of the convergence rates. 

We discuss model specification and the effects on approximation and estimation errors and conclude that the best error bound is achieved using a balance between $k$ and $m$, and inclusion of polynomial terms might render better error bounds. Also, the results of this paper can be generalized to more complex target densities and models with simple modifications to the proofs.

We generalize \citet{jiangtanner1999a} in several directions: (i) we assume one can include polynomial terms of the variables on the GLM1 experts; (ii) we assume the target density is in a $\WK$ class, for $\alpha>0$, instead of $\mathcal{W}^2_{\infty,K_0}$; (iii) we show consistency of the quasi-maximum likelihood estimator for fixed number of experts; (iv) we calculate non-parametric convergence rates of the maximum likelihood estimator; (v) we show non-parametric consistency when the number of experts and the sample size increase; and finally (vi) that using polynomials in the experts one can get better estimation and error bounds. These developments have shed light on the important questions of how many experts should be chosen and how complex the experts themselves should be.

\section*{Acknowledgements}
The authors would like to thank Prof. Marcelo Fernandes and Prof. Marcelo Medeiros for insightful discussion about mixture-of-experts and comments on previous versions of this manuscript.
\bibliographystyle{abbrvnat}
\bibliography{mixture}

\appendix

\section{Showing the convergence rate}\label{ap:conv}
In this appendix we explain and justify the main steps in proving the convergence rate.

One of the drawbacks of working with the Kullback-Leibler divergence is that it is not bounded. An alternative is to use the Hellinger distance. 
\begin{definition}[Hellinger Distance]
	Let $P$ and $Q$ denote two probability measures absolute continuous with respect to some measure $\lambda$. The Hellinger distance between $P$ and $Q$ is given by
	\begin{equation}
		d_h(P,Q) = \left\{\frac{1}{2}\int (\sqrt{\frac{dP}{d\lambda}}-\sqrt{\frac{dQ}{d\lambda}})^2 d\lambda\right\}^{1/2}.
		\label{eq:hellingerP}
	\end{equation}
	Alternatively, the Hellinger distance between two densities $p$ and $q$ with respect to $\lambda$ is given by
	\begin{equation}
		d_h(p,q) = \left\{\frac{1}{2}\int (\sqrt{p}-\sqrt{q})^2 d\lambda\right\}^{1/2}.
		\label{eq:hellingerp}
	\end{equation}
\end{definition}

One can show that if the likelihood ratio is bounded, the KL divergence is bounded by a constant times the square of the Hellinger distance. We use the following result due to \citet{yangbarron2002}, which is presented together with a basic inequality relating the Hellinger distance and Kullback-Leibler divergence.
\begin{lemma}[\citet{yangbarron2002}]
	Let $p_{xy} = dP$ and $\|p_{xy}/f\|_{\infty,{\Omega\times A}}<c_s^2$, for $f\in\mathcal{F}_{m,k}$. Then
	\[
	d_h^2(p_{xy},f)\le KL(p_{xy},f) \leq 2(1+\log c_{s})d_h^2(p_{xy},f),
	\]
	where $d_h(p,f)$ stands for the Hellinger distance between the densities $p$ and $f$ with respect to $\lambda$.
	\label{l:hel-kl}
\end{lemma}

This Lemma implies that the Kullback-Leibler divergence is bounded by the square of the Hellinger distance, and therefore the convergence rate in the square of the Hellinger distance is the same as the convergence rate in the Kullback-Leibler divergence. The only problem is that, in general, the boundedness condition does not hold on the whole set $A$ (the support of $Y$). One could overcome this complication by finding the convergence rate inside some subset of $A$ where the KL divergence is bounded and control the tail probability outside this subset.

Let $S(Y,X)$ denote a scalar function of $(Y_1,X_1'),\dots,(Y_n,X_n)'$ and $B(\beta) = \{y\in A:|y|\le\beta\}$. For every $K\in\R$,
\begin{equation}
	P(S(Y,Y)>K) \le P( \{S(Y,X)>K\}\cap B(\beta))+P( |Y| > \beta).
	\label{eq:prob-ineq}
\end{equation}
If $A$ is bounded, we can choose  $\beta = \ess\sup |A|$, and the second term on the right hand side will be zero. Otherwise, we can take $\beta$ to be large enough such that $P(|Y|>\beta )$ is small or converges to zero at some rate. 

In order to bound the estimation error we shall use results from the theory of empirical processes. The convergence rate theorem presented below is derived for the i.i.d. case, however the same result holds for martingales (see \citet{m-estimation}). 

To control the estimation rate inside a class of functions we have to measure how big is the class. Let $N_B(\varepsilon,\mathcal{F},\|\cdot\|)$ denote the number of $\varepsilon$-brackets\footnote{For a formal definition of \textit{Bracketing Numbers} see \citet{weak-convergence} chapters 2.1 and 2.7} with respect to the distance $\|\cdot\|$, needed to cover the set $\mathcal{F}$ and $H_B(\varepsilon,\mathcal{F},\|\cdot\|)=\log N_B(\varepsilon,\mathcal{F},\|\cdot\|)$ the respective \textit{bracketing entropy}. Moreover, let $const.$ denote some finite universal constant that may change each time it appears, and write $\mathcal{F}^{1/2}_{m,k}=\{\sqrt{f}:\;f\in\mathcal{F}_{m,k}\}$. We show that, under some conditions,
\begin{equation*}
	H_B(\varepsilon,\mathcal{F}_{m,k}^{1/2},\|\cdot\|_{2,{\Omega\times A}}) \leq const.\;(mJ_k+v_m)\log \frac{C}{\varepsilon},
\end{equation*}
for some finite constant $C$ not depending on $\varepsilon$. 

Hence, our first task is to find the bracketing entropy of $\mathcal{F}_{m,k}^{1/2}$. Assumption \ref{a:bound_df} implies that
\begin{align*}
	\frac{\partial\sqrt{ f_{m,k}}}{\partial\zeta'} \le \frac{\sqrt{f_{m,k}}}{2}\left[ c_g(x)',\,\delta_1\,F(x,y)',\,\ldots,\,\delta_m\,F(x,y)' \right],
\end{align*}
where $0\le \delta_i = g_i\pi(h_k(x;\theta_i),y)/f_{m,k} \le 1$ and $\sum_i\delta_i = 1$.

Hence, for any $f_1$ and $f_2$ in some $\mathcal{F}_{m,k}$ indexed respectively by the parameter vectors $\zeta_1$ and $\zeta_2$ in $V_m\times\Theta_{mk}$, 
\begin{equation}
	|\sqrt{f_1} - \sqrt{f_2}| \leq c(x,y)|\zeta_1-\zeta_2|_2,
	\label{eq:lipschitz}
\end{equation}
with $c(x,y)^2 = (\sqrt{f_1}/2)[|F(x,y)'F(x,y)|+|c_g(x)'c_g(x)|]$. Therefore, the square-root densities in $\mathcal{F}_{m,k}$ are Lipschitz in parameters, with Lipschitz function $c(x,y)\in L^2(\lambda,\Omega\times A)$.

\begin{lemma}[Bracketing Entropy]
	Under assumption \ref{a:bound_df}, for any $0<\delta\le1$,
	\begin{equation}
		H_B(\varepsilon,\mathcal{F}_{m,k}^{1/2},\|\cdot\|_2) \le const. (mJ_k+v_m)\log\frac{C}{\varepsilon},
	\end{equation}
	where $C=2\|c(X,Y)\|_{2,\Omega\times A}\diam(V_m\times\Theta_{mk})$; and
	\begin{equation}
		\int_{0^+}^{\delta}\log H_B^{1/2}(u,\mathcal{F}_{m,k}^{1/2},\|\cdot\|_2)du \leq const.(mJ_k+v_m)^{1/2} \delta\log^{1/2}\frac{C}{\delta}
	\end{equation}
	\label{l:bracketing}
\end{lemma}
\begin{proof}
	The first part of the lemma follows from Lemma \ref{l:covering} and equation (\ref{eq:lipschitz}) together with assumption \ref{a:bound_df}.

	The second inequality follows from Lemma \ref{l:int}. If we take $b=\delta$ and $C>e^{\pi}$, we have
	\begin{equation*}
		\int_{0^+}^\delta H_B^{1/2}(u,\mathcal{F}_{m,k}^{1/2},\|\cdot\|_2)du\le const. (mJ_k+v_m)^{1/2}\delta\log^{1/2}\frac{C}{\delta}.
	\end{equation*}
	Proving the lemma.
\end{proof}

Lemma \ref{l:hel-kl} requires the likelihood ratio $|p_{xy}/f_{m,k}^*|$ to be bounded. Next lemma shows the rate of decay for the tail probability $P(|Y|>\beta)$, as a function of $m$ and $J_k$.
\begin{lemma}
	Let $p\in\Pi(\WK)$ and consider densities on $\mathcal{F}_{m,k}$. Then,	under assumption \ref{a:finepart}, in a set with probability not smaller than $1-\eta$;
	\begin{equation}
		\Big\Vert\sup_{p}\inf_{f\in\mathcal{F}_{m,k}}\log\frac{p}{f}\Big\Vert_{\infty,\Omega}\leq const. (mJ_k)^{-\alpha/s}(c\dot{a}_\infty +\dot{b}_\infty) 
		\label{eq:supLR}
	\end{equation}
	where $\eta = \|\var(Y|X=x)/c^2\|_{\infty,\Omega}$, $c$ is some large constant, possibly, depending on $m$ and $k$, and the constants $\dot{a}_\infty$ and $\dot{b}_\infty$ are defined as 
	\begin{align*}
		\dot{a}_\infty &= \ess\sup_{\Omega\times\Theta}|\partial a(h_k(x,\theta))/\partial\theta|, \quad\mbox{and} \\
		\dot{b}_\infty &= \ess\sup_{\Omega\times\Theta}|\partial b(h_k(x,\theta))/\partial\theta|.
	\end{align*}
	\label{l:supLR}
\end{lemma}
\begin{proof}
	Set $\dot{a}_\infty = \ess\sup_{\Omega\times\Theta}|\partial a(h_k(x,\theta))/\partial\theta|$ and $\dot{b}_\infty = \ess\sup_{\Omega\times\Theta}|\partial b(h_k(x,\theta))/\partial\theta|$. For ease of notation, also set $h_{ki}(\cdot)=h_k(\cdot,\theta_i)$. By the convexity of the logarithm,
	\begin{align*}
		\log\frac{p}{f} &\leq \sum_{i=1}^mg_i\log(p/\pi_i)\\
		&=\sum_{i=1}^mg_i(ya(h(x))-a(h_{ki}(x))) + b(h(x))-b(h_{ki}(x)))\\
		&\leq (y\dot{a}_\infty+\dot{b}_\infty) \left[ \sum_{i=1}^m |g_i-I_{\Omega_i}| |h(x)-h_{ki}(x)|+\sum_{i=1}^mI_{\Omega_i}|h(x)-h_{ki}(x)| \right]
	\end{align*}
	Then, by Assumption \ref{a:finepart} and proceeding the same way as in the proof of Theorem \ref{thm:approxrate}, and taking any value $c$
	\begin{equation*}
		\ess\sup_{x\in\Omega, |y|\leq c}\Big\vert\sup_{p}\inf_{f\in\mathcal{F}_{m,k}}\log\frac{p}{f}\Big\vert\leq const. (mJ_k)^{-\alpha/s}(c\dot{a}_\infty +\dot{b}_\infty) 
	\end{equation*}

	The result follows by a simple application of the Chebyschev's inequality.
\end{proof}

The bound on \eqref{eq:supLR} by itself is not enough since we need to relate the function $f^\infty$ satisfying \eqref{eq:supLR} with $f_{m,k}^*$. It follows from Lemma \ref{l:log} that for any $0\le c_l<1$
\[
\log\frac{p}{f^*_{m,k}}\le\frac{1}{(1-c_l)} \log\frac{p}{c_l p + (1-c_l) f_{m,k}^*}.
\]
If we choose $c_l$ small enough, we can find a $c_p$ satisfying
\[
\log\frac{p}{c_l p + (1-c_l) f_{m,k}^*} \le c_p \log \frac{p}{f^\infty}.
\]

Combining this result with the previous lemma we have inside $B(\beta)$
\[
\ess\sup_{\Omega\times\Theta}|p_{xy}/f^*_{m,k}| \le c_\infty\exp\left[ const. \; (mJ_k)^{-\alpha/s}(c\dot{a}_\infty +\dot{b}_\infty)\right],
\]
where $c_\infty = e^{c_p/(1-c_l)}$.

Now we can use theorem 10.13 in \citet{m-estimation} to show the rate of convergence of the Hellinger distance between maximum likelihood estimator and the true density. For sake of completeness the theorem is shown below
\begin{theorem}[Theorem 10.13 in \citet{m-estimation}, pg. 190]
	Let $\hat{f}_{m,k}$ denote the maximum likelihood estimator of $p_{xy}$ over $\mathcal{F}_{m,k}$. Set
	\[
	\bar{\mathcal{F}}_{m,k}^{1/2}(\delta) = \left\{ \sqrt{\frac{f+f^*}{2}}:\, f\in\mathcal{F}_{m,k},\, d_h\left(\frac{f+f^*}{2},f^*\right) \le \delta\right\},
	\] for some fixed $f^*\in\mathcal{F}_{m,k}$ and let $\|p_{xy}/f^*\|_{\infty,\Omega\times A}\leq c_s^2$. Choose
	\[
	\Psi(\delta) \ge \int_{0^+}^{\delta} \log H_B^{1/2}(u,\bar{\mathcal{F}}_{m,k}^{1/2}(\delta),\|\cdot\|_2) du \vee \delta,
	\]
	in such a way that $\Psi(\delta)/\delta^2$ is a non-increasing function of $\delta$. Then, for $\sqrt{n}\delta_n^2\ge const.\, \Psi(\delta_n)$, we have
	\[
	d_h(p_{xy},\hat{f}_{m,k}) = O_p\left( \delta_n +d_h(f_{m,k}^*,p_{xy}) \right).
	\]
	\label{thm:sieves-convrate}
\end{theorem}

\section{Proof of the main results}
\begin{proof}[Proof of theorem \ref{thm:existence}]
	The data generating process of $(x,y)$ and the structure of the model $\mathcal{f}_{m,k}$ is enough to satisfy the measurability assumptions, i.e. it is a weighted sum of measurable functions. also, for any fixed $(x_i,y_i)$, each $\pi_j$ is a continuous function of $\zeta$ $P_{xy}$-almost surely, and the same holds for the $g_j's$, then, $f_{m,k}(X_i,Y_i;\cdot)$ is a continuous function of the parameters $P_{xy}$-a.s. The result follows from theorem 2.12 in \citet{qmle}.

\end{proof}

\begin{proof}[Proof of Theorem \ref{thm:param_consistency}]
	There are different approaches to show consistency of the estimate, we proceed by verifying the conditions of theorem 3.5 in \citet{qmle}.

	The first assumptions regarding the existence of the estimate, are already shown to be satisfied in theorem \ref{thm:existence}. Assumption 3.2, regarding identifiability is satisfied by assumption \ref{a:ident} and \ref{a:uniquemax}. It remains to satisfy assumption 3.1, regarding boundedness and uniform convergence of the log-likelihood function.
	We can show continuity of $\E\log f_{m,k}(X,Y;\zeta)$ by noting that we can interchange integration with limits and a first order Taylor expansion:
\begin{align*}
	\E\left[\log \frac{f_{m.k}(X,Y;\zeta)}{f_{m,k}(X,Y;\zeta-\varepsilon)}\right] & \leq \sup\E\left[\big|\varepsilon'\frac{\partial}{\partial\zeta}f_{m,k}(X,Y;\zeta)\big|\right]\\
	&\leq \sup\E\left[ \big|\frac{\partial}{\partial\zeta}f_{m,k}(X,Y;\zeta)'\frac{\partial}{\partial\zeta}f_{m,k}(X,Y;\zeta)\big| \right]^{1/2}(\varepsilon'\varepsilon)^{1/2},
\end{align*}
which is bounded by lemma \ref{l:fbounds} and by the fact that $\varepsilon$ is arbitrary.

	To show uniform convergence of the likelihood function, we satisfy the conditions of theorem 2 in \citet{jenrich1969}. By assumption, $V_m\times\Theta_{m,k}$ is a compact subset of $\R^{v_m\times mJ_k}$. Measurability and continuity conditions are already satisfied, then it remains to show that $\log f_{m,k}$ is bounded by an integrable function. Note that we can bound the log-likelihood function by:
\begin{equation*}
	\begin{split}
		\Big|\log \frac{f_{m,k}(X,Y;\zeta)}{\varphi(X,Y)}\Big| &= |\log \sum_{i=1}^mg_i(X;\nu)(\pi(h_k(X;\theta_i),y)-c(y))|\\
		&\leq \sum_ig_i(X;\nu)|[a(h_k(X;\theta_i))Y+b(h_k(X;\theta_i))]|\\
		&\leq \max_{1\leq i\leq m}\ess\sup_{x\in\Omega}|a(h_k(x;\theta_i)||Y|+|b(h_k(x;\theta_i))|
	\end{split}
\end{equation*}

Define the bounding function $D(X,Y) = \max_{1\leq i\leq m}\ess\sup_{x\in\Omega}[|a(h_k(X;\theta_i))|$ $\times|Y|+|b(h_k(X;\theta_i))|]$. The function $|h_k(x;\theta)|\leq \sum_{i=1}^{J_k}|\theta_i| < \infty$ because $\max_ix_i = 1$ and $\sum_i |\theta_i| < \infty$, then both $a(h_k)$ and $b(h_k)$ are finite. Thus, it is straightforward to show that $\E D(X,Y)\leq\infty$, given that $\E_{y|x}(Y)\leq\infty$, which is satisfied by assumption about $p(y|x)$.

Then, $\log f_{m,k}(X,Y;\hat\zeta)\rightarrow_{a.s.}\log f_{m,k}(X,Y;\zeta^*)$ as $n\rightarrow\infty$.Therefore, by theorem 3.5 in \citet{qmle}, $\hat\zeta_n\rightarrow \zeta^*$ $P_{XY}$-a.s. as $n\rightarrow\infty$.

\end{proof}

\begin{proof}[Proof of Theorem \ref{thm:approxrate}]
   To bound the approximation rate of the Kullback-Leibler divergence it is enough to bound the upper divergence $\D(f_{m,k},p)$,
	\begin{equation}
		\D(f_{m,k},p) = \int_\Omega \sum_{j=1}^{r_m} g_j(x;\nu)\{h_k(x,\theta_j)-h(x)\}^2dP_x
	   \label{eq:pf1ud}
	\end{equation}
   
	Assumption \ref{a:finepart} ensure the existence of a $\nu_{c_1}$ such that $\max_j\|g_j(\cdot;\nu_{c_1})-I_{Q_j^m}(\cdot)\|_{d,P_x}\leq c_1/r_m\|dP_x/d\lambda\|_{\infty,\Omega}$, where $\|dP_x/d\lambda\|_{\infty,\Omega}$ is finite because $P_x$ has continuous density function with respect to the finite measure $\lambda$ on $\Omega$. Consider
\begin{eqnarray} 
	\D(f_{m,k},p) &\leq \underbrace{\big\|\sum_{j=1}^{r_m} \{g_j(\cdot;v_\varepsilon)-I_{Q_j^m}(\cdot)\}\{h_k(\cdot;\theta_j)-h(\cdot)\}^2\big\|_{1,P_x}\nonumber}_{(A_1)}\nonumber\\
	&\quad + \underbrace{\big\|\sum_{j=1}^{r_m} I_{Q_j^m}(\cdot)\{h_k(\cdot;\theta_j)-h(\cdot)\}^2\big\|_{1,P_x}}_{(A_2)}.
   \label{eq:pf1tri}
\end{eqnarray}

Now we just have to find bounds for both terms in the right hand side of (\ref{eq:pf1tri})($A_1$ and $A_2$). The second term can be written as 
\begin{align*}
(A_2) &= \int\sum_{j=1}^{r_m}I_{Q^m_j}(\cdot)\{h_k(\cdot;\theta_j)-h(\cdot)\}^2dP_x\\
	&= \int\left\{ \sum_{j=1}^{r_m}I_{Q_j^m}(\cdot)\left[h_k(\cdot;\theta_j)-h(\cdot)\right] \right\}^2dP_x,
\end{align*}
where the equality follows from the fact that $I_{Q_j^m}I_{Q_i^m} = I_{Q_j^m}I_{i=j}$, and $\sum_jI_{Q_j^m}(\cdot) = 1$.

If $k<\alpha$, one can choose $\theta_j$ such that $\sup_{x\in Q^m_j}|h_k(x,\theta_j)-h(x)|\leq [K_0/(\bs{k}+1)!] \diam(Q^m_j)^{k+1}$ where $\bs{k}=(k_1,\dots,k_s)$ is an integer vector satisfying $|\bs{k}|=k$ . This claim follows from a Taylor expansion of $h(x)$ around fixed points $x_j\in Q^m_j$ and the fact that $h\in \WK$. Similarly, if $k\ge \alpha$ we can only use the expansion up to $\alpha$ terms. By assumption \ref{a:finepart}, $\sup_j\diam(Q^m_j)\leq1/r_m^{1/s}$.Then 
\begin{equation}
	\sup_j\sup_{x\in Q^m_j}|h_k(x;\theta_j)-h(x)| \leq \frac{c_0}{r_m^{[\alpha\wedge(k+1)]/s}},
	\label{eq:pf1aph}
\end{equation}
where $c_0$ depends only on $K_0$ and $\min(k+1,\alpha)$.

Therefore, $(A_2) \leq c_0^2/r_m^{2[\alpha\wedge(k+1)]/s}$. Note that
\begin{equation}
	\begin{split}
		(A_1) & \leq \sum_{j=1}^m \|\{g_j(\cdot;\nu_\varepsilon)-I_{Q^m_j}(\cdot)\}\cdot \{h_k(\cdot;\theta_j)-h(\cdot)\}^2\|_{1,P_x}\nonumber\\
		&\leq \sup_{j}\sup_{x\in Q^m_j} |h_k(x;\theta_j)-h(x)|^2 \sum_{j=1}^{r_m}\|g_j(\cdot;\nu_\varepsilon)-I_{Q^m_j}(\cdot)\|_{1,P_x}\nonumber\\
		&\leq \frac{c_0^2c_1}{r_m^{2[\alpha\wedge(k+1)]/s}},\nonumber
	\end{split}
\end{equation}
where the last inequality is due to equation (\ref{eq:pf1aph}) and assumption \ref{a:finepart}.

Combining the results for $(A_1)$ and $(A_2)$, 
\begin{equation}
	\D(p,f_{m,k}) \leq \left( \frac{c_0}{r_m^{[\alpha\wedge (k+1)]/s}} \right)^2 (c_1+1).
	\label{eq:pf1bud}
\end{equation}

It follows from lemma \ref{l:updiv} that
\begin{equation}
	KL(p,f_{m,k}) \leq \left( \frac{c_0}{r_m^{[\alpha\wedge(k+1)]/s}} \right)^2M_\infty(c_1+1).
	\label{eq:pf1kl}
\end{equation}

By assumption, $\{r_m\}$ is sub-geometric, then there exists $r_m$ such that $r_m\leq m < r_{m+1}$, and $1/r_m^{2\alpha/s}\geq 1/m^{2\alpha/s} > 1/r_{m+1}^{2\alpha/s}$. Then, $1/(r_mJ_k)^{2\alpha/s}\geq 1/(mJ_k)^{2\alpha/s}>1/(r_{m+1}J_k)^{2\alpha/s}$. By definition of $\mathcal{F}^*_{r_m,k}\subseteq\mathcal{F}^*_{m,k}\subset\mathcal{F}^*_{{r_m+1},k}$. Hence, 
\begin{equation*}
	\begin{split}
		\inf_{f_{m,k}\in\mathcal{F}^*_{m,k}}KL(p,f_{m,k})&\leq\inf_{f_{m,k}\in\mathcal{F}^*_{r_m,k}}KL(p,f_{m,k})\\
		\quad&\leq \frac{M_2c_2}{r_m^{2[\alpha\wedge(k+1)]/s}}\\
		\quad&\leq \frac{c_3}{r_{m+1}^{2[\alpha\wedge(k+1)]/s}}\\
		\quad&\leq \frac{c_3}{m^{2[\alpha\wedge(k+1)]/s}} ,
	\end{split}
\end{equation*}
where $c_2 = M_\infty^2 c_0^2(c_1+1)$ and $c_3 = M_2c_2$ does not depend on $f$. Therefore, 
\begin{equation*}
	\sup_{p\in\Pi(\WK)}\inf_{f_{m,k}\in\mathcal{F}^*_{m,k}}KL(p,f_{m,k}) \leq \frac{c_3}{m^{2[\alpha\wedge(k+1)]/s}}
\end{equation*}

\end{proof}

\begin{proof}[Proof of Theorem \ref{thm:convrate}]
	The the first step is to use equation \eqref{eq:prob-ineq} to bound the convergence rate inside $B(\beta)$ and bound the tail probability. Choose $\beta=c=(mJ_k)^{\alpha/\tau}$, and take $c_s^2 = c_\infty e^{const.[\dot{a}_\infty+\dot{b}_\infty]}$.  It follows from Lemma \ref{l:supLR} and the discussion afterwards that inside $B(\beta)$
	\begin{align*}
		\|p_{xy}/f^*_{m,k}\|_{\infty,\Omega} & \le c_\infty\exp\{const.[\dot{a}_\infty + (mJ_k)^{-\tau/s}\dot{b}_\infty\}\\
		&\le c_\infty\exp\{const [ \dot{a}_\infty+\dot{b}_\infty]\}\\
		&= c_s^2.
	\end{align*}
	This choice of $\beta$ gives us
	\begin{equation}
		\eta = P(|Y|>\beta) = \ess\sup_x\var(Y|X=x)(mJ_k)^{-2\tau/s},
		\label{eq:tail_Y}
	\end{equation}
	by lemma \ref{l:supLR}, and 
	\[
	\ess\sup_x\var(Y|X=x) = \ess\sup_x \frac{\ddot{a}(h(x))\dot{b}(h(x))-\ddot{b}(h(x))\dot{a}(h(x))}{(\dot{a}(h(x)))^3}
	\]
	which is bounded by definition.

	Now, inside $B(\beta)$, we can apply Theorem \ref{thm:sieves-convrate} in the appendix setting $f^*=f^*_{m,k}$. We use Corollary \ref{c:brk-hel} to bound the bracketing number of $\bar{\mathcal{F}}^{1/2}_{m,k}(\delta)$. By Lemma \ref{l:bracketing}
	\begin{equation}
		\int_{0^+}^{\delta}H_B^{1/2}(u,\bar{\mathcal{F}}^{1/2}_{m,k}(\delta),\|\cdot\|_2)\,du\ \le const. (mJ_k+v_m)^{1/2}\delta\log^{1/2}\frac{C}{\delta}.
		\label{eq:brk_number}
	\end{equation}
	Since $C\propto \sqrt{mJ_k+v_m}$, we can choose $\Psi(\delta) \propto (mJ_k)^{1/2}\delta\log^{1/2}\left( \frac{(mJ_k+v_m)^{1/2}}{\delta} \right)$. This choice of function satisfies $\Psi(\delta)/\delta^2$ is non-increasing, and we can take $\delta_n = (mJ_k+v_m)^{1/2}(\sqrt{\log n/n})$. In fact, this choice of $\delta_n$ gives us
	\begin{align*}
		\sqrt{n}\delta_n &\ge const. \Psi(\delta_n)\\
		\delta_n &\ge const. \sqrt{\frac{mJ_k+v_m}{n}} \log^{1/2}\frac{(mJ_k+v_m)^{1/2}}{\delta_n}\\
		& = const. \sqrt{\frac{mJ_k+v_m}{n}}\left( \frac{1}{\sqrt{2}}\log^{1/2}n - \frac{1}{2}\log\frac{\log n}{\sqrt{2}} \right)\\
		&= \frac{const.}{\sqrt{2}}\delta_n - \frac{const.}{2}\sqrt{\frac{mJ_k+v_m}{n}}\log \frac{\log n}{\sqrt{2}}. 
	\end{align*}
	
	Hence the convergence rate in Hellinger distance is given by:
	\begin{equation*}
		d_h(\hat{f}_{m,k},p_{xy}) = O_p\left( d_h( f^*_{m,k},p_{xy})+(mJ_k+v_m)^{1/2}\sqrt{\frac{\log n}{n}} \right).
	\end{equation*}

	Our choice of $\beta$ allows to apply Lemma \ref{l:hel-kl} to obtain
	\begin{equation*}
		KL(\hat{f}_{m,k},p_{x,y}) = O_p\left( KL( f^*_{m,k},p_{x,y})+(mJ_k+v_m)\frac{\log n}{n}\right).
\end{equation*} 
	We use Theorem \ref{thm:approxrate} to conclude that, inside $B(\beta)$,
	\begin{equation}
		KL(p_{x,y},\hat{f}_{m,k}) = O_p\left( \frac{1}{(m)^{2\tau/s}} + (mJ_k+v_m)\frac{\log n}{n} \right).
		\label{eq:convrate-B}
	\end{equation}

	Combining this rate inside $B(\beta)$ with \eqref{eq:tail_Y}, we arrive in our result \eqref{eq:convrate}.

	We achieve the best rate \eqref{eq:optimrate} by taking $m \propto (\log n/n)^{-s/2\tau+s}$ and substituting this rate in \eqref{eq:convrate}.

\end{proof} 

\begin{proof}[Proof of Theorem \ref{thm:op_convrate}]
	The proof is parallel to \ref{thm:convrate}, with just some small changes to lemma \ref{l:bracketing}. More precisely, since we have a unique $f^*_{m,k}$ for each $(m,k)$, we can find the bracketing number inside $\mathcal{F}_{m,k}^{1/2}(\delta)$. The argument is the same with the only difference that inside $\mathcal{F}_{m,k}^{1/2}(\delta)$, $\diam(V_m\times\Theta_{mk})= const. \,\delta$, removing the log-term on the right hand side of \eqref{eq:brk_number}. 

	This change allows to choose $\delta_n = \sqrt{(mJ_k+v_m)/n}$, removing the $\log n$ term of the rate.

\end{proof}

\begin{proof}[Proof of Proposition \ref{pro:choosemandk}]
I.  Under the constraint $m\xi^s=C$, $U=C/n+ (C\xi^{-s})^{-2(\xi\wedge \alpha)/s}$.  Consider two cases: When $\xi>\alpha$, $U$   obviously increases with $\xi$. So the optimal $\xi_o\leq \alpha$.  When $\xi\leq \alpha$, computing the derivative and we know that the function $U$ is minimized at $\xi=(C^{1/s}/e)$. When this point is to the right of $\alpha$, the function $U$ decreases for all $\xi\leq\alpha$. So we obtain $\xi_o= \alpha$.  When $(C^{1/s}/e)$ is to the left of $\alpha$, the minimum is achieved at $(C^{1/s}/e)$ and $\xi_o= (C^{1/s}/e)$. Combining these we obtain  the minimizer $\xi_o=\alpha\wedge(C^{1/s}/e)$.

II. and III. They are straightforward from (\dag).

\end{proof}

\section{Auxiliary Results}
In the next lemma, we use the notation $\partial_\theta=\partial/\partial\theta$, $\partial_{\theta\theta'}=\partial^2/\partial\theta\partial\theta'$, $a_j = a(h_k(x;\theta_j))$, $\dot{a}_j = \partial_\theta a_j$, $\ddot{a}_j=\partial_{\theta\theta'}a_j$ and so on.

\begin{lemma}
	Let $f\in\mathcal{F}_{m.k}$. Under assumption \ref{a:bound_df}
	\begin{itemize}
		\item $\E|\log f|\leq\infty$
		\item $\E|\nabla \log f| \leq \infty$
		\item $\E|\nabla \log f|_2^2 \leq \infty$
		\item if we further assume \ref{a:ident} and \ref{a:uniquemax}, then $\E|\nabla^2 \log f| \leq \infty$ and is nonsingular at $\zeta^*$.
	\end{itemize}
	\label{l:fbounds}
\end{lemma}
\begin{proof}
	This theorem is proved by calculate the derivatives and bounding it.

	First note that $a_j$ and $b_j$ are continuous differentiable functions of $h_k(x;\theta_j)$. Since $|h_k(x;\theta_j)|\leq |\theta_j|<\sqrt{J_k}|\theta_j|_2<\infty$ for any fixed $k$, then both $a_j$ and $b_j$ are also bounded. The same reasoning can be applied to $\dot{a}_j$, $\dot{b}_j$, $\ddot{a}_j$ and $\ddot{b}_j$. Also, by definition, $\E|y|^p<\infty$ for any $p\ge0$.Then
	
	\[
	\E\log f \leq \E[\log \sum_{j=1}^m g_j \pi_j] \leq \E|\max_j[ya_j+b_j+c(y)]| < \infty.
	\]

	Let $\delta_j = g_je^{ya_j+b_j}p_x/f\leq 1$ and $c^*=\max_{j}\|\partial_\nu\log g_j\|_{\infty,\Omega}$, then
	
	\begin{equation*}
		\begin{split}
			\E|\partial_{\theta_j}\log f| &= \E|\delta_j (y\dot{a}_j+\dot{b}_j) x|< \infty,\\
			\E|\partial_\nu\log f| &= \E\big|\frac{\sum \dot{g}_j e^{ya_j+b_j}}{f}\big|\leq m c^* < \infty.
		\end{split}
	\end{equation*}

	The same follows for $\E|\partial_\theta\log f|_2^2$, $\E|\partial_\nu\log f|_2^2$ and $\E|\partial_\theta\log f\partial_\nu\log f|$. Let $\dot{c}^* = \|\partial_\nu \log \dot{g}_j\|_{\infty,\Omega}$, and choose any vector $\alpha$ with appropriate dimensions satisfying $\alpha'\alpha=1$ then
	
	\begin{equation*}
		\begin{split}
			\E\alpha'|\partial_{\theta_j\theta_j'}\log f|\alpha &= \E\alpha|\delta_j(1-\delta_j)(y\dot{a}_j+\dot{b}_j)^2xx'+\delta_j(y\ddot{a}_j+\ddot{b}_j)xx'|\alpha\\
			&\le 0.25 \E|y\dot{a}_j+\dot{b}_j|^2 + \E\max_j|y\ddot{a}_j+\ddot{b}_j| < \infty,\\
			\E\alpha'|\partial_{\theta_j\theta_k'}\log f|\alpha &= \E\alpha'|-\delta_j\delta_k(y\dot{a}_k+\dot{b}_k)(y\dot{a}_j+\dot{b}_j)xx'|\alpha \\
			&\le \E|(y\dot{a}_k+\dot{b}_k)(y\dot{a}_j+\dot{b}_j)| < \infty,\\
			\E\alpha'|\partial_{\theta_j\nu'}\log f|\alpha &= \E\alpha'\big|\frac{e^{ya_j+b_j}(y\dot{a}_j+\dot{b}_j)x\dot{g}'p_x}{f}\big|\alpha\\
			&\leq \E\alpha'|\delta_j(y\dot{a}_j+\dot{b}_j)x\bs{1_{v_m}}'|\alpha c^* \\
			&\leq \E | y\dot{a}_j+\dot{b}_j|c^*< \infty,\\
			\E\alpha'|\partial_{\nu\nu'}\log f|\alpha  &= \E\alpha'|\frac{\sum_j\ddot{g}_je^{ya_j+b_j}p_x}{f}-\frac{\sum_j\dot{g}_j e^{ya_j+b_j}p_x}{f}\frac{\sum_j\dot{g}_j' e^{ya_j+b_j}p_x}{f}|\alpha\\
			&\leq c^*|\dot{c}^*|+{c^*}^2 < \infty.
		\end{split}
	\end{equation*}
	
	Since $\zeta^*$ is a maximizer of $\E\log f$ over $\mathcal{F}_{m,k}$, $\E|\nabla^2 \log f| $ has to be non-negative definite. Assumption \ref{a:uniquemax} tells us it is also invertible, therefore $\E|\nabla^2 \log f|$ is positive definite.
	
\end{proof}

\begin{lemma}
	For any $0<a<b\le1$ and a positive constant $C$,
	\begin{equation}
		\int_a^b\log^{1/2}\frac{u}{C}du \le b\left( \sqrt{\pi}\log^{1/2}\frac{C}{b} \right).
		\label{eq:l-int}
	\end{equation}
	\label{l:int}
\end{lemma}
\begin{proof}
	For any $0<a<b\le1$,
	\begin{align*}
		\int_a^b \log^{1/2}\frac{u}{C}du &= 2C\int_{\log^{1/2}(C/b)}^{\log^{1/2}(C/a)}v^2e^{-v^2}dv\\
		&\le 2C\int_{\log(C/b)}^\infty t^{1-3/2}e^{-t}dt\\
		&= C\Gamma(3/2,\log(C/b)\\
		&\le b\left( \sqrt{\pi}+\log^{1/2}(C/b) \right),
	\end{align*}
	where the last inequality follows from
	\begin{align*}
		\Gamma(3/2,x)&=\frac{1}{2}\Gamma(1/2,x)+x^{1/2}e^{-x}\\
		&=\sqrt{\pi}\Phi(-\sqrt{2x})+x^{1/2}e^{-x}\\
		&\le (\sqrt{\pi}+x^{1/2})e^{-x}. 
	\end{align*}
\end{proof}

\begin{lemma}
	Let $p$ and $q$ denote two positive densities. For any $0\le c_l< 1$,
	\begin{equation}
		\log\frac{p}{q}\le\frac{1}{1-c_l}\log\frac{p}{c_l p + (1-c_l) q}.
		\label{eq:l-log}
	\end{equation}
	\label{l:log}
\end{lemma}
\begin{proof}
	By the convexity of the logarithm we have
	\begin{align*}
		\log\frac{p}{c_l p + (1-c_l) q} &= -\log(c_l+(1-c_l)q/p)\\
		&\ge c_l(-\log 1) + (1-c_l)(-\log q/p)\\
		&=(1-c_l)\log\frac{p}{q}
	\end{align*}
\end{proof}

\begin{lemma}[Lemma 4.2 in \citet{m-estimation}] 
	We have, for $f_1$, $f_2$ and some $f^*$, that
	\begin{equation}
		\sqrt{2}\,d_h\left( \frac{f_1+f^*}{2},\frac{f_2+f^*}{2} \right) \le d_h(f_1,f_2).
		\label{eq:hel-lin}
	\end{equation}
	\label{l:hel-lin}
\end{lemma}

This lemma is similar to lemma 4.2 in \citet{m-estimation}, with the only difference being that we consider an arbitrary $f^*$ and \citet{m-estimation} considers $f^*$ to be the true density. The proof remains unchanged.

\begin{corollary}
	Let $\bar{\mathcal{F}}^{1/2}_{m,k}(\delta)$ be as in theorem $\ref{thm:sieves-convrate}$, and
	\[
	\mathcal{F}_{m,k}^{1/2}(\delta) = \left\{ \sqrt{f}:\, f\in\mathcal{F}_{m,k},\; d_h(f,f^*)\le \delta\right\}.
	\]
	We have that $N_B(\varepsilon,\bar{\mathcal{F}}_{m,k}^{1/2}(\sqrt{2}\delta),\|\cdot\|_2) \le N_B(\varepsilon,\mathcal{F}^{1/2}_{m,k}(\delta),\|\cdot\|_2)$.
	\label{c:brk-hel}
\end{corollary}
\begin{proof}
	The proof follows from lemma \ref{l:hel-lin}, taking $f_1 = f$ and $f_2=f^*=f^*$. We have that an $(\sqrt{2}\varepsilon)$-bracket net for $\bar{\mathcal{F}}_{m,k}^{1/2}(\sqrt{2}\delta)$ is also an $\varepsilon$-bracket net for $\mathcal{F}_{m,k}^{1/2}(\delta)$, all with respect to $\|\cdot\|_2$.
\end{proof}

The next lemma provides a bound on the bracketing number of functional classes that are Lipschitz in a parameter.

\begin{lemma}[Theorem 2.7.11 in \citet{weak-convergence}]
	Let $\mathcal{F} = \{f_t:\,t\in T\}$ be a class of functions satisfying
	\[
	|f_s(x) -f_t(x)| \le d(s,t)\,F(x),
	\]
	for some metric $d$ on $T$, function $F$ on the sample space and every $x$. Then for any norm $\|\cdot\|$, 
	\begin{equation}
		N_B(2\varepsilon\|F\|,\mathcal{F},\|\cdot\|) \le N(\varepsilon,T,d),
		\label{eq:covering}
	\end{equation}
	where $N(\varepsilon,T,d)$ is the$\varepsilon$-covering number of $T$ with respect to the metric $d$.
	\label{l:covering}
\end{lemma}

It is straightforward to see that if we set $\dim(T) = d$, $c_T = \diam(T)$, and $C = 4\|F\|_2$, 
\begin{equation*}
	N_B(\varepsilon,\mathcal{F},\|\cdot\|_2) \le \left(\frac{c_TC}{\varepsilon}  \right)^{d}
	\label{}
\end{equation*}

\end{document}